\documentclass[11pt]{article}

\usepackage[T1]{fontenc}
\usepackage[utf8]{inputenc}
\usepackage[margin=1in]{geometry}
\usepackage{amsmath,amssymb,amsfonts}
\usepackage{amsthm}
\usepackage{mathrsfs}
\usepackage{graphicx}
\usepackage{booktabs}
\usepackage{xcolor}
\usepackage{textcomp}
\usepackage{enumitem}
\usepackage[title]{appendix}
\usepackage[numbers,sort&compress]{natbib}
\usepackage{url}
\usepackage[colorlinks=true,linkcolor=blue,citecolor=blue,urlcolor=blue]{hyperref}

\DeclareUnicodeCharacter{00A0}{~}
\DeclareUnicodeCharacter{2011}{-}
\DeclareUnicodeCharacter{2013}{--}
\DeclareUnicodeCharacter{2014}{---}
\DeclareUnicodeCharacter{2212}{-}

\newcommand{\R}{\mathbb{R}}
\newcommand{\CC}{\mathbb{C}}
\newcommand{\eps}{\varepsilon}
\newcommand{\dd}{\,d}
\DeclareMathOperator{\dist}{dist}

\newcommand{\backmatter}{}
\newcommand{\bmhead}[1]{\section*{#1}}

\newtheoremstyle{thmstyleone}
  {3pt}{3pt}{\itshape}{}{\bfseries}{.}{.5em}{}
\newtheoremstyle{thmstyletwo}
  {3pt}{3pt}{\normalfont}{}{\bfseries}{.}{.5em}{}
\newtheoremstyle{thmstylethree}
  {3pt}{3pt}{\normalfont}{}{\bfseries}{.}{.5em}{}

\theoremstyle{thmstyleone}
\newtheorem{theorem}{Theorem}
\newtheorem{proposition}[theorem]{Proposition}
\newtheorem{lemma}[theorem]{Lemma}
\newtheorem{corollary}[theorem]{Corollary}

\theoremstyle{thmstyletwo}

\newtheorem{assumption}[theorem]{Assumption}

\newtheorem{remark}{Remark}

\theoremstyle{thmstylethree}
\newtheorem{definition}{Definition}

\title{A Priori Integral Persistent Excitation in Conservative Polynomial ODEs with Higher-Order Interactions}
\author{Aleksandr Semenov\thanks{Control of complex systems, Institute for problems in mechanical engineering of the Russian Academy of Science, Boljshoy prospekt V.O., Saint Petersburg 199178, Russia. Email: \texttt{aleksandr.semenov.dm@gmail.com}}
\and Alexander Fradkov\thanks{Control of complex systems, Institute for problems in mechanical engineering of the Russian Academy of Science, Boljshoy prospekt V.O., Saint Petersburg 199178, Russia. Email: \texttt{fradkov@mail.ru}}\thanks{These authors contributed equally to this work.}}
\date{}

\begin{document}
\maketitle

\begin{abstract}
The paper proposes an approach for verifying integral persistent excitation, which is important in problems of parameter identification and adaptive control in nonlinear dynamical systems. The approach works for conservative polynomial ODEs a priori without knowledge of the parameters. Rigorous proofs of the corresponding theorems are provided. An example of a nonlinear dynamical system with higher-order interactions and the application of the proposed method to it are analyzed. The proof of the main result is formalized in the Lean formal verification language.
\end{abstract}


\section{Introduction}

In physical problems, as well as in problems of chemistry, mechanics, biology, and economics, it is necessary to estimate parameters of dynamical systems \cite{FradkovMiroshnikNikiforov1999}. This is precisely the main question of parametric identification theory. Nevertheless, parameter estimation algorithms are often invented outside the cybernetics journals, for example, in physics papers \cite{KralemannCimponeriuRosenblumPikovskyMrowka2007,GoelMaitraMontroll1971}, whenever a corresponding need arises.

Besides inventing a method for parameter estimation, it is important to be certain that the obtained estimates are accurate and correct. The performance criterion for identification algorithms is usually the persistent excitation (PE) condition and its various modifications: weak PE, interval excitation, etc. \cite{WangEfimovBobtsov2020,NarendraAnnaswamy1987,PanteleyLoriaTeel2001,AeyelsSepulchre1994}.

However, a paradox arises here. First, the PE condition is not easy to prove even for known parameter values, because one needs to verify an inequality for the integral of a function along the trajectory. Second, verifying the PE condition is of interest before the parameters are identified; consequently, one has to work with a system where the parameters are unknown. Coping with these difficulties is hard, and often even in very strong papers the PE condition is introduced as an assumption \cite{RybalkoFradkov2023,XingNaCostaCastelloGao2020,KralemannCimponeriuRosenblumPikovskyMrowka2008} or the question of identifiability is not raised \cite{KralemannCimponeriuRosenblumPikovskyMrowka2007}. \emph{A priori} ways to verify the PE condition are poorly developed \cite{PadoanScarciottiAstolfi2017}, and publications on this topic are rare \cite{PadoanScarciottiAstolfi2017,SemenovFradkov2024,SemenovFradkov2025}. The aim of the present paper is to propose a method to solve these problems for a certain class of systems.

When identifying parameters of complex autonomous nonlinear systems with higher-order interactions, the question becomes especially acute. In classical approaches, PE is often provided by an input signal (for instance, a sum of sinusoids is fed into the input); however, for a complex system such an intervention may be too crude, especially since the influence of external inputs on systems with higher-order interactions cannot be considered completely understood \cite{jafari2024synchronization}. For example, for ecological systems such an approach would mean regular mass culling of animals alternating with their abundant feeding.

In the present paper, we restrict ourselves to conservative dynamical systems governed by polynomial ODEs. Furthermore, for simplicity and clarity of exposition (similarly to \cite{PadoanScarciottiAstolfi2017}) we will consider the PE condition for the state variable vector.

The paper is organized as follows. Section \ref{sec:preliminaries} introduces affine integral excitation and its scalar form. Section \ref{sec:problem} formulates the problem setting for a polynomial vector field on a compact invariant set. Section \ref{sec:main_result_2} presents the main result on the absence of sticking to hyperplanes and affine IPE. Sections \ref{sec:triad_model_2} and \ref{sec:triad_application_2} consider a resonant triad model and the application of the main theorem to it. Section \ref{sec:discussion_2} Concludes the article and is devoted to the development of work ideas. The proof of the main theorem is deferred to the appendix.

\section{Preliminaries}
\label{sec:preliminaries}

\subsection{Affine Integral Excitation}

\begin{definition}[Affine integral excitation]
Let $z:[0,\infty)\to\R^n$ be a measurable bounded trajectory. Define the affine regressor as the column
\[
    \phi_z(t)=\begin{bmatrix} z(t) \\ 1 \end{bmatrix}\in\R^{n+1}.
\]
We call the trajectory $z(t)$ \emph{affinely integrally exciting} if there exists a number $\alpha>0$ such that
\begin{equation}\label{eq:ipe_definition_2}
    \liminf_{T\to\infty}
    \frac1T\int_0^T \phi_z(t)\phi_z(t)^\top\dd t
    > \alpha I_{n+1}>0.
\end{equation}

\end{definition}

\begin{lemma}[Scalar form of affine IPE]
\label{lem:scalar_ipe_equivalence_2}
Condition \eqref{eq:ipe_definition_2} is equivalent to the following property: there exists a number $\alpha>0$ such that for any affine hyperplane
\[
    a\cdot x+b=0,
    \qquad \|a\|^2+b^2=1, \quad a \in \R^n , b \in \R
\]
the estimate
\begin{equation}\label{eq:ipe_scalar_2}
    \liminf_{T\to\infty}
    \frac1T\int_0^T |a\cdot z(t)+b|^2\dd t
    >\alpha
\end{equation}
holds.

\end{lemma}
Geometrically, this condition means that the trajectory does not stick to any affine hyperplane as $t \rightarrow \infty$.

Affine IPE and its consequences are sufficient conditions for identifiability of constant parameters in many parameter estimation and adaptive control algorithms \cite{AeyelsSepulchre1994,NarendraAnnaswamy1987}.

\section{Problem Statement}
\label{sec:problem}

Consider the system
\begin{equation}\label{eq:polynomial_system_2}
    \dot x=f_\xi(x),\qquad x\in\R^n,
    \quad \xi\in\Xi\subset\R^m,
\end{equation}
where the components of $f_\xi$ are polynomials in $x$ and the parameters enter the right‑hand side linearly:
\begin{equation}\label{eq:linear_parameterization_2}
    \dot{x}_i=\sum_{j}\xi_{ij}p_{ij}(x).
\end{equation}
Here $p_{ij}(x)$ are monomials. 

This formulation covers many models of interest, including generalized Lotka--Volterra system, oscillator networks, resonant cascades, and other complex systems.

In the sequel, we will drop the subscript $\xi$ wherever this causes no confusion and it is clear from the context which parameter value is meant; in particular, we will write $f$, $K$, $\mu$, $\Phi^t$, $M$, and $Q$ instead of $f_\xi$, $K_\xi$, $\mu_\xi$, $\Phi_\xi^t$, $M_\xi$, and $Q_\xi$.

\begin{assumption}[Conditions]
\label{ass:main_2}
Let $K\subset\R^n$ be a compact invariant set of system \eqref{eq:polynomial_system_2}. Assume that the following conditions hold.
\begin{enumerate}[label=(A\arabic*)]
    \item $\Phi^t(K) \subset K$ for all $t\in\R$.
    \item On $K$ there exists an invariant probability measure $\mu$ that is absolutely continuous with respect to the Lebesgue measure, and
    \[
        d\mu=\rho(x)\dd x,
        \qquad
        0<c\le \rho(x)\le C<\infty
        \quad \text{on }K.
    \]
    \item For the matrix
    \[
        M(x)=\bigl[f(x)\mid L_f f(x)\mid\dots\mid L_f^{n-1}f(x)\bigr]
    \]
    the polynomial
    \[
        Q(x)=\det M(x) \not\equiv 0
    \]
    is nondegenerate: there exists a point $x_*\in K$ such that $Q(x_*)\ne0$.
\end{enumerate}
\end{assumption}

Notice that this list of conditions is not overly restrictive. In particular, any Hamiltonian system satisfies the first two conditions. Indeed, let $W(x)$ be the Hamiltonian of the system; then the set $K = \{x:c_1\le W(x) \le c_2\}$, where $c_1,c_2 > 0$, is compact and invariant. By Liouville's theorem, a Hamiltonian system preserves the Lebesgue measure. The third condition is technical and can usually be guaranteed for almost all parameter values. For an easy check, the following statement can be applied.

\begin{proposition}[Checking nondegeneracy for almost all parameters]
\label{prop:param_nondegeneracy_2}
Consider the family \eqref{eq:polynomial_system_2}--\eqref{eq:linear_parameterization_2} and fix a point $x_*\in\R^n$. If there exists a parameter $\xi^0\in\Xi$ such that
\[
    Q_{\xi^0}(x_*)\ne0,
\]
then the set of parameters
\[
    \Sigma_{x_*}=\{\xi\in\Xi:Q_\xi(x_*)=0\}
\]
has zero Lebesgue measure in $\R^m$.
\end{proposition}

It is proposed to give answers further in measure‑theoretic terms: “for almost all”, “up to a set of measure zero”, and the like. The reason is that for any somewhat complicated nonlinear system one can usually choose parameters in a certain degenerate way so that some conclusion fails. Declaring the parameters unknown, one cannot completely exclude degenerate cases (without adding a large number of complicated, hard‑to‑verify specific conditions); the best one can claim is that the chance of these “bad” cases is zero.

Thus, we formulate the main question of the paper as follows: suppose Assumption \ref{ass:main_2} hold for system \ref{eq:polynomial_system_2}; is it true that for $\mu$-almost every initial condition $x\in K$ the trajectory $t\mapsto\Phi^t(x)$ is affinely integrally exciting?

\section{Main Result}
\label{sec:main_result_2}

 The answer to the question of the previous section turns out to be positive.

For an affine hyperplane $H\subset\R^n$ set
\[
    U_\eps(H)=\{y\in K:\dist(y,H)<\eps\}.
\]

\begin{theorem}[Absence of sticking to hyperplanes and affine IPE]
\label{thm:main_nonsticking_2}
Let system \ref{eq:polynomial_system_2} satisfy Assumption \ref{ass:main_2}. Then there exists an invariant set $X_0\subset K$ of full measure, $\mu(X_0)=1$, such that for every $x\in X_0$ and every affine hyperplane $H\subset\R^n$ the fraction of time spent in the $\epsilon$-neighborhood of the hyperplane $U_\eps(H)$ tends to zero:
\begin{equation}\label{eq:main_nonsticking_2}
    \lim_{\eps\to0}\limsup_{T\to\infty}
    \frac1T\int_0^T\mathbf 1_{U_\eps(H)}(\Phi^t x)\dd t=0.
\end{equation}

\end{theorem}
\begin{proof}
    The proof of the theorem is given in the Appendix. For simplicity of verification, a reference to a Lean formalization of this proof is also provided there.
\end{proof}
\begin{corollary}
   Consequently, for every $x\in X_0$ the trajectory $t\mapsto\Phi^t(x)$ is affinely integrally exciting:
\begin{equation}\label{eq:main_ipe_2}
    \liminf_{T\to\infty}
    \frac1T\int_0^T \phi_z(t)\phi_z(t)^\top\dd t
    >0.
\end{equation}
\end{corollary}

As mentioned in the introduction, IPE for the state vector is considered mainly for clarity; the main theorems (the Birkhoff–Khinchin theorem \cite{Birkhoff1931,CornfeldFominSinai1982} and the Novikov–Yakovenko theorem \cite{NovikovYakovenko1999}) used in the proof (see Appendix) allow one to generalize the result of the paper to an arbitrary polynomial regressor and to apply the theorem to the regressor that is used in a concrete work on parameter identification or adaptive control.

Let us now show how this theorem can be applied to a concrete model.

\section{Illustrative Example}
\label{sec:triad_model_2}

As an illustrative example we take a simple ODE system: reduced three‑wave interactions \cite{kaup1979space}.

\begin{equation}\label{eq:triad_complex_2}
    i\dot A_1=\kappa_1\overline{A_2}A_3,
    \qquad
    i\dot A_2=\kappa_2\overline{A_1}A_3,
    \qquad
    i\dot A_3=\kappa_3A_1A_2,
\end{equation}
where $A_1,A_2,A_3\in\CC$, $\kappa\in D \subset \R^3$.

Write $A_j=x_j+iy_j$. Then \eqref{eq:triad_complex_2} defines a real polynomial system in $\R^6$:
\begin{equation}\label{eq:triad_real_2}
\begin{aligned}
\dot x_1&=\kappa_1(x_2y_3-y_2x_3),&
\dot y_1&=-\kappa_1(x_2x_3+y_2y_3),\\
\dot x_2&=\kappa_2(x_1y_3-y_1x_3),&
\dot y_2&=-\kappa_2(x_1x_3+y_1y_3),\\
\dot x_3&=\kappa_3(x_1y_2+y_1x_2),&
\dot y_3&=-\kappa_3(x_1x_2-y_1y_2).
\end{aligned}
\end{equation}

In the sequel we will work with the real vector
\[
    z=(x_1,y_1,x_2,y_2,x_3,y_3)\in\R^6,
    \qquad
    r_j(z)=x_j^2+y_j^2,
    \quad j=1,2,3.
\]

\begin{proposition}[Basic properties of the three-waves model]
\label{prop:triad_basic_2}
For every fixed $\kappa$, system \eqref{eq:triad_real_2} is polynomial, has zero divergence in $\R^6$, and therefore preserves the Lebesgue measure. Moreover, the functions
\begin{equation}\label{eq:triad_MR_2}
    I_{13}^{\kappa}(z)=\frac{x_1^2+y_1^2}{\kappa_1}+\frac{x_3^2+y_3^2}{\kappa_3},
    \qquad
    I_{23}^{\kappa}(z)=\frac{x_2^2+y_2^2}{\kappa_2}+\frac{x_3^2+y_3^2}{\kappa_3}
\end{equation}
are first integrals of the real system \eqref{eq:triad_real_2}.
\end{proposition}

\subsection{Application of the Main Theorem to the Three-waves model}
\label{sec:triad_application_2}

Choose numbers
\[
    0<a_1<b_1,
    \qquad
    0<a_2<b_2,
    \qquad
    \delta>0,
\]
and consider the set
\begin{equation}\label{eq:triad_band_2}
    K_{\kappa}=\left\{z\in\R^6:
    a_1\le I_{13}^{\kappa}(z)\le b_1,
    \quad
    a_2\le I_{23}^{\kappa}(z)\le b_2,
    \quad
    |I_{13}^{\kappa}(z)-I_{23}^{\kappa}(z)|\ge\delta
    \right\}.
\end{equation}

To check nondegeneracy we introduce the polynomial
\begin{equation}\label{eq:triad_delta_2}
    Q(z,\kappa)=
    \det\bigl[z^{(1)}(0;z,\kappa)\mid z^{(2)}(0;z,\kappa)
    \mid\dots\mid z^{(6)}(0;z,\kappa)\bigr],
\end{equation}
where $z(t)\in\R^6$ is a trajectory of the real system \eqref{eq:triad_real_2} and
\[
    z^{(j)}(0;z,\kappa)=\left.\frac{d^j}{dt^j}z(t)\right|_{t=0}
\]
for the solution with initial condition $z(0)=z$. Equivalently, the columns in \eqref{eq:triad_delta_2} equal
\[
    f(z,\kappa),\quad L_f f(z,\kappa),\quad \dots,\quad L_f^5 f(z,\kappa).
\]

At the point
\begin{equation}\label{eq:triad_test_point_2}
    z_*=(1,2,3,4,5,6),
    \qquad
    \kappa^0=\left(\frac{1}{3},\frac{1}{3},\frac{1}{3}\right)
\end{equation}
a direct calculation gives
\begin{equation}\label{eq:triad_delta_value_2}
    Q(z_*,\kappa^0)\ne0.
\end{equation}
Hence, for the fixed point $z_*$ the condition of Proposition \ref{prop:param_nondegeneracy_2} is satisfied. Therefore, except for a set of parameters $\kappa\in D$ of zero measure, the nondegeneracy condition of Assumption \ref{ass:main_2} holds on any band $K_\kappa$ for which $z_* \in K_\kappa$. In this case, by Theorem \ref{thm:main_nonsticking_2}, for almost all initial conditions $z_0\in K_\kappa$ the trajectory of the real model does not stick on average to any affine hyperplane in $\R^6$ and is affinely integrally exciting.

\section{Conclusion}\label{sec:discussion_2}

In this paper we have proposed an approach for verifying integral persistent excitation for conservative polynomial ODEs \emph{a priori}, without knowledge of the parameters. Rigorous proofs of the corresponding statements are provided.
Using the three‑wave interaction model as an example, we have shown how the method can be applied to systems with higher‑order interactions. The reasoning given for the illustrative example can be repeated without much difficulty for substantially more complex models, such as resonant cascades.
The authors hope that the community will find these results interesting and will use them for their own models. Future research will be aimed at applying the approach to various concrete problems of identification and adaptive control of complex systems, as well as at extending the results to the case of systems with dissipation.

\backmatter

\bmhead{Acknowledgements}

\section*{Declarations}

Not applicable


\begin{appendices}

\section{Lemma \ref{lem:scalar_ipe_equivalence_2}}\label{secA1}

\begin{proof}[Proof of Lemma \ref{lem:scalar_ipe_equivalence_2}]
Set
\[
    G_T=\frac1T\int_0^T\phi_z(t)\phi_z(t)^\top\dd t.
\]
For $u=(a,b)\in\R^{n+1}$, $\|u\|^2=\|a\|^2+b^2=1$, we have
\[
    u^\top\phi_z(t)\phi_z(t)^\top u
    =(u^\top\phi_z(t))^2
    =|a\cdot z(t)+b|^2.
\]
Hence
\[
    u^\top G_Tu
    =\frac1T\int_0^T |a\cdot z(t)+b|^2\dd t.
\]
The matrix inequality in \eqref{eq:ipe_definition_2} is understood in the sense of quadratic forms. Therefore, if \eqref{eq:ipe_definition_2} holds, then multiplying it from the left by $u^\top$ and from the right by $u$ gives \eqref{eq:ipe_scalar_2}. Conversely, if \eqref{eq:ipe_scalar_2} holds for all normalized $u=(a,b)$, then the last identity gives the same lower bound for all normalized quadratic forms of $G_T$. This is precisely \eqref{eq:ipe_definition_2}.
\end{proof}

\section{Proof of Parametric Nondegeneracy}
\label{app:param_nondegeneracy_2}

\begin{proof}[Proof of Proposition \ref{prop:param_nondegeneracy_2}]
From the linear parametrization \eqref{eq:linear_parameterization_2} it follows that the components of $f_\xi(x)$ are polynomials in the joint variables $(x,\xi)$. Therefore $L_{f_\xi} f_\xi$, $L_{f_\xi}^2 f_\xi$, $\dots$, $L_{f_\xi}^{n-1}f_\xi$ are also polynomials in $(x,\xi)$. Consequently,
\[
    Q_\xi(x)=\det\bigl[f_\xi(x)\mid L_{f_\xi} f_\xi(x)\mid\dots\mid L_{f_\xi}^{n-1}f_\xi(x)\bigr]
\]
is a polynomial in $(x,\xi)$, and for a fixed $x_*$ the function $\xi\mapsto Q_\xi(x_*)$ is a polynomial in the parameters. Since for some $\xi^0$ we have $Q_{\xi^0}(x_*)\ne0$, this polynomial is not identically zero. Its zero set is a proper algebraic set and has zero Lebesgue measure. This proves the claim.
\end{proof}

\section{Technical Lemmas for the Main Theorem}
\label{app:technical_lemmas_2}
A Lean formalization (mathlib 4.28.0) of the proof of this theorem can be found at \url{https://github.com/Sashkasem/IPE_proof_lean}.

Throughout this appendix we assume that Assumption \ref{ass:main_2} holds, the parameter $\xi$ is fixed, and we use the convention of dropping the index introduced above. Set
\[
    M(x)=\bigl[f(x)\mid L_f f(x)\mid\dots\mid L_f^{n-1}f(x)\bigr],
    \qquad
    Q(x)=\det M(x),
\]
\[
    Z=\{x\in K:Q(x)=0\},
    \qquad
    W_\delta=\{x\in K:|Q(x)|<\delta\}.
\]

\begin{lemma}[Properties of the singular set]
\label{lem:singular}\label{lem:singular_set_2}
\begin{enumerate}
    \item If a trajectory lies entirely in a hyperplane $H$, then it lies entirely in $Z$.
    \item $\lambda(Z)=0$, hence $\mu(Z)=0$, and $\mu(W_\delta)\to0$ as $\delta\to0$.
\end{enumerate}
\end{lemma}

\begin{proof}
(1) Let $H=\{y:\langle v,y\rangle=c\}$, $v\ne0$. Differentiating the identity
\[
    \langle v,\Phi^t(x)\rangle\equiv c
\]
at zero, we obtain
\[
    \langle v,L_f^k f(x)\rangle=0,
    \qquad k=0,\dots,n-1.
\]
In other words, the nonzero vector $v$ is orthogonal to all columns of the matrix $M(x)$. Hence the columns of $M(x)$ are linearly dependent and $Q(x)=0$. This argument applies not only to the initial point but also to any other point of the same trajectory. Therefore, if a trajectory lies entirely in a hyperplane, it lies entirely in $Z$.

(2) By Assumption \ref{ass:main_2} there exists a point $x_*\in K$ such that $Q(x_*)\ne0$. Hence the polynomial $Q$ is not identically zero. A nonzero polynomial vanishes on a set of zero Lebesgue measure, so $\lambda(Z)=0$. The absolute continuity of the measure $\mu$ yields $\mu(Z)=0$. Since $W_\delta\downarrow Z$ as $\delta\downarrow0$, continuity of measure from above gives $\mu(W_\delta)\to\mu(Z)=0$.
\end{proof}

\subsection*{Step 0: description of the set $X_0$}

\begin{proposition}[Step 0]
\label{prop:step0_X0_2}
There exists an invariant set of full measure $X_0\subset K$ such that for every $x\in X_0$
\begin{equation}\label{eq:step0_X0_2}
    \lim_{\delta\to0}\lim_{T\to\infty}
    \frac1T\int_0^T\mathbf 1_{W_\delta}(\Phi^t x)\dd t=0.
\end{equation}
\end{proposition}

\begin{proof}
For each $\delta>0$, by the Birkhoff–Khinchin theorem \cite{Birkhoff1931,CornfeldFominSinai1982}, for almost every $x\in K$ the limit
\[
    b_\delta(x)=\lim_{T\to\infty}\frac1T\int_0^T\mathbf 1_{W_\delta}(\Phi^t x)\dd t
\]
exists. Here the Birkhoff–Khinchin theorem is applied to the indicator of the neighborhood of $Z$, i.e., to the function $g(x)=\mathbf 1_{W_\delta}(x)$. Moreover,
\[
    \int_K b_\delta(x)\dd\mu
    =\int_K \mathbf 1_{W_\delta}(x)\dd\mu
    =\mu(W_\delta).
\]

Since $W_\delta$ decreases as $\delta\downarrow0$, the function $\delta\mapsto b_\delta(x)$ is nonincreasing almost everywhere. Define
\[
    b_0(x)=\lim_{\delta\to0}b_\delta(x)=\inf_{\delta>0}b_\delta(x).
\]
By the monotone convergence theorem,
\[
    \int_K b_0(x)\dd\mu
    =\lim_{\delta\to0}\int_K b_\delta(x)\dd\mu
    =\lim_{\delta\to0}\mu(W_\delta)=0,
\]
because $\mu(W_\delta)\to\mu(Z)=0$. Consequently, $b_0(x)=0$ for $\mu$-almost all $x$.

Set
\[
    X_0=\left\{x\in K:
    \lim_{\delta\to0}\lim_{T\to\infty}
    \frac1T\int_0^T\mathbf 1_{W_\delta}(\Phi^t x)\dd t=0
    \right\}.
\]
Then $\mu(X_0)=1$. The invariance of $X_0$ follows from the invariance of time averages with respect to a shift of the initial moment.

Thus, the set $X_0$ is $K$ without the set $Z$ and without those points for which the corresponding statement of the Birkhoff–Khinchin ergodic theorem does not hold. For any $x\in X_0$ and any $\eta>0$ there exists $\delta>0$ such that
\[
    b_\delta(x)=\lim_{T\to\infty}\frac1T\int_0^T\mathbf 1_{W_\delta}(\Phi^t x)\dd t<\eta.
\]
\end{proof}

\subsection*{Dynamics outside $W_\delta$}

Fix a hyperplane $H=\{l(x)=0\}$ with $l(x)=\langle a,x\rangle-c$ and $\|a\|=1$. For $x\in K$ along the trajectory we will write
\[
    x(t)=\Phi^t(x),
    \qquad
    x^{(k)}(0)=\frac{d^k}{dt^k}\Phi^t(x)\bigg|_{t=0}=L_f^{k-1}f(x).
\]
From the compactness of $K$ and the polynomial nature of the field it follows that all the derivatives used below are uniformly bounded on $K$. Denote a common upper bound by $M$.

\begin{lemma}[Uniform separation of derivatives]
\label{lem:sep}\label{lem:derivative_separation_2}
For each $\delta>0$ there exists $c_\delta>0$ such that for all $x_0\in K\setminus W_\delta$ and all unit normals $a$,
\[
    \max_{1\le k\le n}|\langle a,x^{(k)}(0)\rangle|\ge c_\delta.
\]
\end{lemma}

\begin{proof}
Consider the continuous function
\[
    \Psi(x_0,a)=\max_{1\le k\le n}|\langle a,x^{(k)}(0)\rangle|
\]
on the compact set $(K\setminus W_\delta)\times S^{n-1}$. We prove by contradiction that its minimum cannot be zero. If the minimum were zero, then for some pair $(x_0,a)$ the nonzero vector $a$ would be orthogonal to all vectors
\[
    x^{(1)}(0),\dots,x^{(n)}(0),
\]
i.e., to all columns of $M(x_0)$. Then $\det M(x_0)=0$, hence $x_0\in Z\subset W_\delta$, contradicting $x_0\notin W_\delta$. Therefore the minimum is positive; set $c_\delta:=\min\Psi>0$.
\end{proof}

\begin{lemma}[Exit time]
\label{lem:exit}\label{lem:exit_time_2}
For every $T_*>0$ there exists $\eps_*(\delta)>0$ such that if $x_0\notin W_\delta$, then the trajectory cannot stay continuously in $U_\eps(H)$ longer than $T_*$ for all $\eps\le\eps_*(\delta)$.
\end{lemma}

\begin{proof}
Choose $T_*$. By Lemma \ref{lem:sep} there exists a derivative of order $k$ such that
\[
    |\langle a,x^{(k)}(0)\rangle|\ge c_\delta,
\]
where $a$ is the normal to $H$. Denote
\[
    C_k=\langle a,x^{(k)}(0)\rangle.
\]
Without loss of generality we assume that this derivative is positive. As noted before, all derivatives can be bounded on the compact set by a common constant $M$. Hence the $k$-th derivative keeps its sign for at least time $C_k/M$. Consider the time interval $[0,\tau_k]$, where
\[
    \tau_k=\min\{C_k/M,T_*/k\},
\]
and see how the $(k-1)$-st derivative changes during this time.

On this interval it is monotone and increases by at least
\[
    \Delta C_{k-1}=\int_0^{\tau_k}(C_k-Mt)\dd t.
\]
This expression has a lower bound that depends only on $T_*$, $k$, $c_\delta$, that is, essentially only on $\delta$ and our choice of $T_*$. Next observe that either the $(k-1)$-st derivative was less than $-\frac12\Delta C_{k-1}$ at time $0$, or it becomes greater than $\frac12\Delta C_{k-1}$ at time $\tau_k$. In any case its modulus can be bounded below by the constant
\[
    C_{k-1}=\frac12\int_0^{\tau_k}(C_k-Mt)\dd t=\frac12\Delta C_{k-1}.
\]

Repeating this argument by induction we conclude that with $\eps_*(\delta)=C_0$ the trajectory leaves the $\eps$-neighborhood of $H$ in time not exceeding the sum of $k$ terms, each of which is at most $T_*/k$ (either $\tau_k$ or $0$). That is, in time not exceeding $T_*$. Moreover, this estimate is global.
\end{proof}

\begin{remark}[Important remarks on Lemmas \ref{lem:sep} and \ref{lem:exit}]
\begin{enumerate}
    \item The estimate in Lemma \ref{lem:sep} is uniform over all $x_0\in K\setminus W_\delta$ and all unit normals $a$.
    \item In Lemma \ref{lem:exit} it is important only that at the initial moment $x_0\notin W_\delta$; the subsequent dynamics may enter $W_\delta$ and leave it many times.
    \item The number $T_*$ can be chosen arbitrarily small and then adjusted to other estimates.
\end{enumerate}
\end{remark}

\subsection*{Estimate of the number of intersections with surfaces}

To obtain an estimate of the average time, one must exclude the situation where the trajectory frequently leaves and immediately enters $U_\eps(H)$. Therefore it is important to bound from above the number of intersections with the boundaries $\{l=\pm\eps\}$ on each time block. The Novikov–Yakovenko theorem is used for this purpose.

\begin{proposition}[Novikov–Yakovenko theorem, autonomous case]
\label{prop:novikov_yakovenko}\label{prop:novikov_yakovenko_2}

Let a polynomial ODE
\[
    \dot{x}=p(x), \qquad x \in \mathbb{R}^n,
\]
be given, where the $p_i(x)$ have degree at most $d$. Assume that the absolute values of the coefficients of these polynomials are bounded by $r$.

Let $\Gamma: I \to \mathbb{R}^n$ be a trajectory on an interval $I \subset \mathbb{R}$ that is contained in a ball of radius $r$. The length of the interval $I$ also does not exceed $r$.

Then for any affine hyperplane $H \subset \mathbb{R}^n$ the number of isolated intersection points of $\Gamma$ with $H$ (counting multiplicities) is bounded by a constant $\Omega(r, n, d) \le (2 + r)^{B(n,d)}$ that depends only on $n$, $d$, and $r$.
\end{proposition}

\begin{lemma}[Uniform bound on intersections]
\label{lem:intersect}\label{lem:intersect_2}

For a fixed $T_0$ and small $\delta,\eps$ there exists a universal constant $N$ such that for any hyperplane $H$ the number of intersections of the trajectory with the surfaces $\{l=\pm\eps\}$ on any interval of length $T_0$ does not exceed $N$.
\end{lemma}

\begin{proof}
First fix a hyperplane $H$. Consider the two surfaces given by the linear polynomials $l(x)\pm\eps$. Since $K$ is compact, $\eps$ is small, and the hyperplane $H$ intersects the compact set $K$, the coefficients of all these polynomials are bounded by a constant that does not depend on the concrete small values of $\delta$ and $\eps$. By the Novikov–Yakovenko theorem \ref{prop:novikov_yakovenko}, the number of zeros of each of these polynomials along the trajectory on an interval of length $T_0$ is bounded by some constant $N$.

It remains to observe that choosing another hyperplane changes nothing essential: the Novikov–Yakovenko theorem is insensitive to such a transition — its estimate depends on the degree of the polynomials and the dimension of the space, and on the compact set the coefficients of the corresponding linear polynomials remain uniformly bounded after normalization.
\end{proof}

\section{Proof of the Main Theorem}
\label{app:proof_main_theorem_2}

\begin{proof}[Proof of Theorem \ref{thm:main_nonsticking_2}]
First we use the set $X_0$ from Proposition \ref{prop:step0_X0_2}. Take arbitrary $x\in X_0$ and an affine hyperplane $H=\{l=0\}$.

\textbf{Idea of the proof.}
We prove by contradiction. Assume that for the chosen $x\in X_0$ and $H$ the fraction of time in $U_\eps(H)$ does not tend to zero as $\eps\to0$. For a chosen small $\delta$ we partition time into blocks. In each block the time in $U_\eps(H)$ is split into two parts:
\[
    \text{(i) inside }W_\delta,
    \qquad
    \text{(ii) outside }W_\delta.
\]
For the first part we already have a small average estimate from Step 0. For the second part we need to show that for sufficiently small $\eps$ the average time outside $W_\delta$ can also be made arbitrarily small.

\textbf{Step 1: choice of $\delta$.}
Assume the contrary: there exist $\gamma>0$ and a sequence $\eps_k\to0$ such that
\begin{equation}\label{eq:contradiction_assumption_2}
    \limsup_{T\to\infty}\frac1T\int_0^T\mathbf 1_{U_{\eps_k}(H)}(\Phi^t x)\dd t\ge\gamma.
\end{equation}
By the definition of $X_0$ we can choose $\delta>0$ so that
\begin{equation}\label{eq:small_W_delta_2}
    b_\delta(x)=\lim_{T\to\infty}\frac1T\int_0^T\mathbf 1_{W_\delta}(\Phi^t x)\dd t<\frac{\gamma^2}{64}.
\end{equation}
For this $\delta$ we further fix the exit time estimate from Lemma \ref{lem:exit} and the constant $c_\delta$ from Lemma \ref{lem:sep}.

\textbf{Step 2: blocks and bad blocks.}
Choose the block length $T_0=1$. Consider a large time $T=mT_0$. The number $T$ is chosen large enough so that the average
\[
    \frac1T\int_0^T\mathbf 1_{W_\delta}(\Phi^t x)\dd t
\]
is sufficiently close to its limit from \eqref{eq:small_W_delta_2}. Partition $[0,mT_0]$ into blocks
\[
    I_j=[(j-1)T_0,jT_0],
    \qquad j=1,\dots,m.
\]
Denote the fraction of time in $W_\delta$ inside a block by
\[
    v_j=\frac1{T_0}\int_{I_j}\mathbf 1_{W_\delta}(\Phi^t x)\dd t,
\]
and the fraction of time in $U_\eps(H)$ by
\[
    w_j=\frac1{T_0}\int_{I_j}\mathbf 1_{U_\eps(H)}(\Phi^t x)\dd t.
\]
We call a block $I_j$ bad if $v_j\ge\gamma/8$.

Since $x\in X_0$, for the chosen $\delta$ we have
\[
    \lim_{m\to\infty}\frac1m\sum_{j=1}^m v_j<\frac{\gamma^2}{64}.
\]
Hence, for all sufficiently large $m$,
\[
    \frac1m\sum_{j=1}^m v_j<\frac{\gamma^2}{64}.
\]
If $m_{\mathrm{bad}}$ is the number of bad blocks, then
\begin{equation}\label{eq:bad_blocks_2}
    \frac{m_{\mathrm{bad}}}{m}\cdot\frac\gamma8
    \le\frac1m\sum_{j=1}^m v_j
    <\frac{\gamma^2}{64},
    \qquad
    \frac{m_{\mathrm{bad}}}{m}<\frac\gamma8.
\end{equation}

\textbf{Step 3: fixing the time grid and the number of intersections $N$.}
By the Novikov–Yakovenko theorem and Lemma \ref{lem:intersect}: for the chosen $T_0$ and sufficiently small $\eps$ there exists a universal constant $N$ such that on each block of length $T_0$ the number of intersections of the trajectory with the surfaces $\{l=\pm\eps\}$ does not exceed $N$.

Consequently, the number of entries into $U_\eps(H)$ on one block is at most $N/2$. Here an entry into $U_\eps(H)$ means the initial moment of each continuous interval of stay in $U_\eps(H)$.

\textbf{Step 4: choice of $T_*$.}
Choose $T_*>0$ so small that
\begin{equation}\label{eq:Tstar_choice_2}
    \frac{N}{2}T_*<\frac\gamma8.
\end{equation}
This is possible because $T_*$ in Lemma \ref{lem:exit} can be taken arbitrarily small. Next fix $\eps\le\eps_*(\delta)$ so that the bound on the number of intersections and the exit time estimate are applicable.

More precisely, the order of choice is as follows: first pick a small range for $\eps$, then obtain the corresponding estimate $N$ from the Novikov–Yakovenko theorem, then choose $T_*$ so that \eqref{eq:Tstar_choice_2} holds, and then, if necessary, decrease the admissible $\eps$ according to Lemma \ref{lem:exit}. The estimate $N$ does not increase from such a decrease.

\textbf{Step 5: good block, case 1.}
Consider a good block, i.e., a block for which $v_j<\gamma/8$. Consider the next entry into $U_\eps(H)$.

First case: the entry occurs outside $W_\delta$. At the moment of entry $x_0\notin W_\delta$, hence by Lemma \ref{lem:exit} the trajectory leaves $U_\eps(H)$ in time at most $T_*$. It does not matter whether it later enters $W_\delta$ during this interval or not: Lemma \ref{lem:exit} gives the answer independently of a subsequent entry into $W_\delta$.

\textbf{Step 6: good block, case 2.}
Second case: the entry occurs inside $W_\delta$. Let the trajectory first stay in $W_\delta$ for time $\tau_i$, then leave $W_\delta$ and, as in the first case, leave $U_\eps(H)$ in time at most $T_*$. Therefore for each such entry we obtain a contribution not exceeding $\tau_i+T_*$.

Summing over all entries, we obtain a sum of the form
\[
    \sum_i(\tau_i+T_*).
\]
There are at most $N/2$ entries, so the contribution of the parts after leaving $W_\delta$ does not exceed $(N/2)T_*$. On a good block the sum of all times $\tau_i$ spent inside $W_\delta$ does not exceed $\gamma/8$. Hence,
\begin{equation}\label{eq:good_block_2}
    w_j\le\frac\gamma8+\frac{N}{2}T_*\le\frac\gamma4.
\end{equation}

\textbf{Preventing fast oscillation.}
To estimate the total time in $U_\eps(H)$ via the number of entries, one must exclude too frequent entries and exits. This is precisely why Lemma \ref{lem:intersect} is used: the number of intersections with the boundaries $\{l=\pm\eps\}$ on each block is bounded above, so the trajectory cannot oscillate arbitrarily many times between the boundaries.

\textbf{Step 7: final assembly.}
For good blocks we obtain $w_j\le\gamma/4$; for bad blocks we use the crude bound $w_j\le1$. Then by \eqref{eq:bad_blocks_2}
\[
    \frac1m\sum_{j=1}^m w_j
    \le\frac{m-m_{\mathrm{bad}}}{m}\cdot\frac\gamma4
    +\frac{m_{\mathrm{bad}}}{m}\cdot1
    <\frac\gamma4+\frac\gamma8
    =\frac{3\gamma}{8}.
\]
Thus the fraction of time spent in $U_\eps(H)$ does not exceed $3\gamma/8$. This contradicts the assumption \eqref{eq:contradiction_assumption_2}. Consequently, the fraction of time in $U_\eps(H)$ tends to zero as $\eps\to0$, i.e., \eqref{eq:main_nonsticking_2} is proved.

It remains to deduce affine IPE. If the scalar estimate \eqref{eq:ipe_scalar_2} does not hold, then there exist $T_j\to\infty$ and normalized pairs $(a_j,b_j)$, $\|a_j\|^2+b_j^2=1$, such that
\[
    \frac1{T_j}\int_0^{T_j}|a_j\cdot \Phi^t(x)+b_j|^2\dd t\to0.
\]
Passing to a subsequence, we assume $(a_j,b_j)\to(a,b)$. If $a=0$, then $|b|=1$, which is impossible. If $a\ne0$, then for the hyperplane $H=\{y:a\cdot y+b=0\}$ property \eqref{eq:main_nonsticking_2} gives a small average fraction of time in $U_\eta(H)$ for sufficiently small $\eta>0$, while outside this neighborhood $|a\cdot \Phi^t(x)+b|\ge \|a\|\eta$. This yields a positive lower bound for the limit inferior of the mean square, a contradiction. Hence \eqref{eq:ipe_scalar_2}, and therefore \eqref{eq:main_ipe_2}, holds.
\end{proof}

\end{appendices}

\end{document}